\documentclass[12pt,a4paper]{article}
\usepackage[utf8]{inputenc}

\usepackage{amsmath,amssymb,amsthm,cite,graphicx,geometry,todonotes,latexsym,amsfonts,enumerate,booktabs,diagbox,makecell,mathtools}

\usepackage{hyperref}
\hypersetup{colorlinks,allcolors=blue}
\newcommand{\footremember}[2]{%
    \footnote{#2}
    \newcounter{#1}
    \setcounter{#1}{\value{footnote}}%
}

\newcommand{\Fix}{\mathrm{Fix}\,}
\newcommand{\Ker}{\mathrm{Ker}\,}
\newcommand{\conv}{\mathrm{conv}\,}
\newcommand{\aff}{\mathrm{aff}\,}
\newcommand{\bdry}{\mathrm{bdry}\,}

\newcommand{\cone}{\mathrm{cone}\,}
\newcommand{\geo}{\mathrm{geo}\,}
\newcommand{\N}{\mathbb{N}}
\newcommand{\R}{\mathbb{R}}

\newcommand{\Id}{\mathrm{Id}}
\newcommand{\Hi}{\mathcal{H}}

\newtheorem{theorem}{Theorem}

\newtheorem{proposition}[theorem]{Proposition}
\newtheorem{lemma}[theorem]{Lemma}

\newtheorem{corollary}[theorem]{Corollary}

\theoremstyle{definition}
\newtheorem{example}[theorem]{Example}
\newtheorem{definition}[theorem]{Definition}

\title{Finite Convergence of Circumcentered-Reflection Method on Closed Polyhedral Cones in Euclidean Spaces}
\author{Hongzhi Liao\footremember{alley}{School of Mathematics and Statistics, UNSW Sydney, Australia, e-mail: hongzhi.liao@unsw.edu.au}}

\begin{document}

\maketitle

\begin{abstract}
    The Circumcentered Reflection Method (CRM) is a recently developed projection method for solving convex feasibility problems. It offers preferable convergence properties compared to classic methods such as the Douglas–Rachford and the alternating projections method. In this study, our first main theorem establishes that CRM can identify a feasible point in the intersection of two closed convex cones in $\R^2$ from any starting point in the Euclidean plane. We then apply this theorem to intersections of two polyhedral sets in $\R^2$ and two wedge-like sets in $\R^n$, proving that CRM converges to a point in the intersection from any initial position finitely. Additionally, we introduce a modified technique based on CRM, called the Sphere-Centered Reflection Method. With the help of this technique, we demonstrate that CRM can locate a feasible point in finitely many iterations in the intersection of two proper polyhedral cones in $\R^3$ when the initial point lies in a subset of the complement of the intersection's polar cone. Lastly, we provide an example illustrating that finite convergence may fail for the intersection of two proper polyhedral cones in $\R^3$ if the initial guess is outside the designated set.
\end{abstract}


\section{Introduction}

Projection methods, such as the Douglas–Rachford and alternating projections methods, are widely applied algorithms for solving (convex) feasibility problems, that is, finding a common point in the intersection of finitely many (convex) sets. Feasibility problems are fundamental in science and engineering; see \cite{bauschke1996projection} and \cite{combettes1993foundations} for examples. The Circumcentered Reflection Method (CRM), a modification of the Douglas–Rachford method, was first proposed by Behling, Cruz, and Santos in \cite{BCS2018-1} to address feasibility problems, initially applied to two subspaces in $\R^n$. Shortly thereafter, they extended CRM to finitely many affine subspaces in $\R^n$ and proved its linear convergence rate \cite{BCS2018-2}, showing that it outperforms the alternating projections method in terms of convergence rate \cite{arefidamghani2021circumcentered}. Bauschke, Ouyang, and Wang further provided sufficient conditions for the linear convergence of CRM \cite{bauschke2020circumcentered}. Additionally, Behling, Cruz, and Santos demonstrated that this approach could be generalized to convex sets in $\R^n$ with the same convergence rate, see \cite{iusem2022circumcentered} and \cite{BCS2021}.

\par However, fewer results address the finite convergence performance of CRM. Behling, Cruz, and Santos demonstrated that CRM finds the projection of any given point onto the intersection of finitely many hyperplanes in 
$\R^n$ in a single step \cite{BCS2020}. In Hilbert spaces, Ouyang proved that for cases involving two hyperplanes, two half-spaces, or a combination of one hyperplane and one half-space, CRM reaches either the best approximation or a feasible point of the intersection within at most three steps, depending on the linear relation of the normal vectors of the half-space or hyperplane \cite{ouyang2020finite}.

\par In this paper, we build on ideas from \cite{dao2023douglas} to prove that CRM finds a feasible point in the intersection of two closed convex cones in $\R^2$ in three steps (see Theorem \ref{main}). By applying this theorem along with \cite[Theorem 3]{BCS2021}, we demonstrate finite convergence for the intersection of two polyhedral sets in $\R^2$ and two wedge-like sets in $\R^n$, as shown in Theorem \ref{poly_finite} and Corollary \ref{cor:wedge2}. In $\R^3$, we prove that CRM can locate a feasible point in the intersection of two proper polyhedral cones in finitely many steps if the initial guess lies in the ``finite convergence zone" (see Theorem \ref{thm:R3poly}). Finally, we provide an example illustrating that finite convergence does not hold in $\R^3$ for two proper polyhedral cones if the initial guess is outside this zone (see Example \ref{counterexample}).

\par The structure of this paper is as follows. In Section \ref{Pl}, we introduce the notation, auxiliary results, and the definition of the Circumcentered Reflection Method. The main theorem for $\R^2$ is presented in Section \ref{F}. In Section \ref{P}, we apply this theorem to polyhedral sets in $\R^2$ and to higher-dimensional cases where the reflections lie in the same plane. To conclude, in Section \ref{FR3} we introduce the Sphere-Centered Reflection Method (see Definition \ref{SRM}) and use it to analyze the finite convergence performance for two proper polyhedral cones in $\R^3$. This section also includes an example showing that finite convergence of CRM does not hold for two proper polyhedral cones in $\R^3$ (and higher dimensions) if the initial guess is outside the ``finite convergence zone."

\section{Preliminaries}\label{Pl}

\subsection{Notations and auxiliary results}

In this paper, we use \(\Hi\) to stand for a \textit{real Hilbert spaces}, \(\langle \cdot \,, \cdot \rangle\) for the \textit{inner product} and \(\|\cdot\|\) for the \textit{norm} associated with the inner product in $\R^n$. Given a point $x \in \Hi$ and $r>0$, \(B_r(x)\) is the \textit{closed ball} in \(\Hi\) centered at $x$ with radius $r$. Let $S$ be a nonempty subset of $\Hi$, then $\aff S$ is \textit{the smallest affine subspace} of $\Hi$ containing $S$, the \textit{interior} and \textit{boundary} of $S$ is \(\mbox{int}\, S\) and $\bdry S$ respectively. 

\par For two nonempty subsets $A,B \subseteq \Hi$, we define $A+B \coloneqq\{x+y \,|\, x\in A,y\in B\}$ and $A-B \coloneqq \{x-y \,|\, x\in A,y\in B\}$, known as the \textit{Minkowski sum and subtraction} respectively. We use $A\oplus B$ to denote the \textit{direct sum} of $A$ and $B$, $A \perp B$ to represent $A$ is \textit{perpendicular} to $B$, i.e., for all $a \in A$ and $b \in B$, we have $\langle a\,, b \rangle = 0$. 

\par Let $x\in \Hi$ and $C\subseteq \Hi$ be nonempty, the \textit{distance} of $x$ to $C$ is $d_C(x)\coloneqq{\rm inf}_{y \in C} \|x-y\|$. 
If there exists a point $p \in C$ such that \(\|x-p\|=d_C(x)\), we call $p$ is a \textit{best approximation} to $x$ from $C$, or a \textit{projection} of $x$ onto $C$. If every point in \(\Hi\) has a unique projection onto \(C\), we say $C$ is a \textit{Chebyshev set}. In this case, we let \(P_C(x)\coloneqq p\) represent the projection. Recall there is a famous theorem about the Chebyshev set and the property of projections, see \cite[Theorem 3.16]{BauschkeHeinz2017}: 

\begin{theorem}[Projection theorem]\label{ProjectionTheorem}
If $C$ is a nonempty closed convex subset of $\Hi$, then C is a Chebyshev set. And for every $x,p \in \Hi$, the following statements are equivalent:
\begin{enumerate}[(i)]
    \item \(p=P_C(x)\);
    \item \(p \in C\), for all \(y \in C\), \(\langle y-p \,, x-p \rangle \leq 0.\)
\end{enumerate}
\end{theorem}

\par Suppose \(C \subseteq \Hi\) is a nonempty closed convex set. The \textit{reflection operator} associated with \(C\) is defined as \(R_C \coloneqq 2P_C - \Id\), where \(\Id\) is the \textit{identity mapping}.

Let \(S\) be a subset of \(\Hi\), \(S\) is a \textit{closed cone} if it is closed and satisfies 
\[S=\bigcup_{\lambda \geq 0}\lambda S, \;{\rm where}\; \lambda S = \{\lambda x \,|\, x \in S\} \;\mbox{for a given}\; \lambda \in \R.\]
The \textit{conic hull} of $S$ contains all conical combinations of elements in $S$, that is 
    \[\cone (S)\coloneqq \left\{\sum_{i=1}^k \alpha_i x_i \,\bigg|\, x_i \in S, \alpha_i \geq 0, k\in \N \right\}.\]
    
Assume \(S\) is a nonempty subset of \(\Hi\), the \textit{polar cone} of \(S\) is given by
    \[S^{\ominus}\coloneqq\{x \in \Hi \,|\,  \langle x ,y \rangle \leq 0 \;\,\mbox{for any}\;\, y \in S\}.\]
    
For an operator $T:\Hi \to \Hi$, the \textit{kernel} of $T$ is
\[\Ker T\coloneqq \{x \in \Hi \,|\, T(x)=0\},\]
and its \textit{fixed-point set} is 
\[\Fix T\coloneqq\{x \in \Hi \,|\, T(x)=x\}.\]

Given point $x \in \Hi$, if there exists $n \in \N$ such that $T^n(x) \in \Fix T$, we say $T$ \textit{converges finitely} from $x$.
\par Let $x,y \in \Hi$ and $x \neq y$, the \textit{closed line segment} between points $x$ and $y$ is
    \[[x,y] \coloneqq \{(1-\alpha)x+\alpha y \,|\, 0 \leq \alpha \leq 1\}.\]
Assume $x$ is nonzero, the \textit{ray} generated by $x$ is
    $$[x] \coloneqq \{\lambda x \,|\, \lambda \geq 0\}.$$

\subsection{Circumcentered-reflection method (CRM)}

In this subsection, we introduce the circumcentered-reflection method.
The \textit{circumcenter} of three given non-collinear points $x,y,z \in \Hi$, which is denoted as $C (x,y,z)$, means $C(x,y,z)$ is equidistant to $x,y,z$ and \(C(x,y,z) \in \aff \{x,y,z\}\).


\begin{definition}
Let $A,B \subseteq \Hi$ be closed convex sets, $x\in \Hi$ and \(R_A, R_B\) be reflection operators associated with \(A, B\) respectively. The \textit{circumcentered-reflection operator} $C_T$ is defined as
\begin{equation*}
    C_T(x)\coloneqq C(x, R_A(x),R_B(R_A(x))).
\end{equation*}
\end{definition}

Figure \ref{fig:CRM} gives a geometrical intuition when \(A,B\) are closed convex cones in \(\R^2\), we can see $C_T(x) = 0$.

\begin{figure}[!ht]
    \centering
    \includegraphics[width = 13cm]{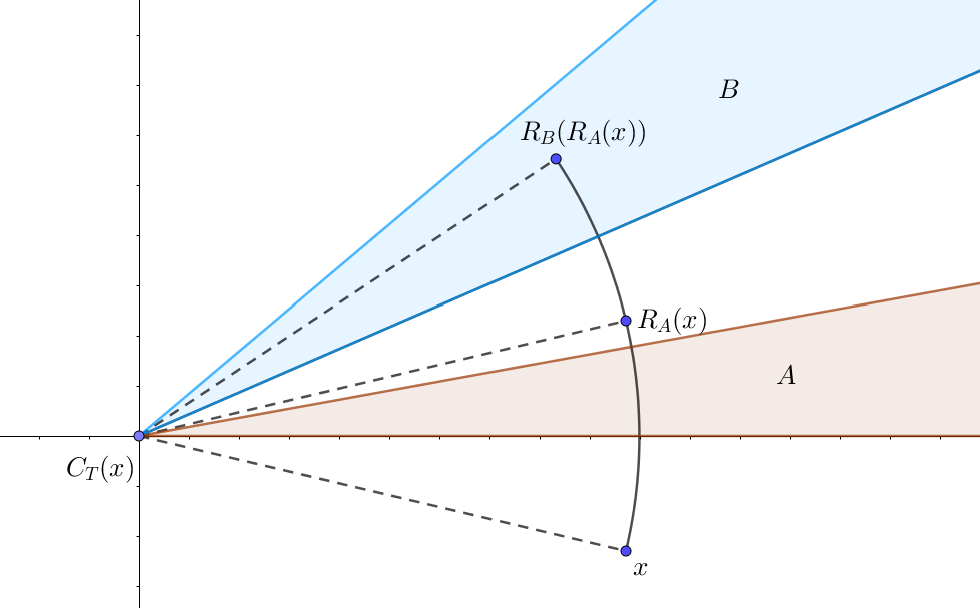}
    \caption{Geometrical interpretation of CRM when \(A,B\) are closed convex cones in \(\R^2\).}
    \label{fig:CRM}
\end{figure}

It is necessary to check whether $C_T$ is well-defined for given $A$ and $B$. Here we briefly specify two cases:
\begin{itemize}
    \item if $A,B$ are subspaces of $\R^n$, Behling, Cruz, and Santos has stated $C_T$ is well-defined in \cite[Introduction]{BCS2018-1}.
    
    \item if $A,B$ are closed convex sets of $\R^n$, $C_T(x)$ may not exist for some $x \in \R^n$. To fix this issue, according to \cite{BCS2021}, if we define \(X\coloneqq\{(x,x)\in \R^{2n} \,|\, x \in \R^n\}\) and \(Y\coloneqq A\times B\), then
    \begin{equation}\label{XY}
        x^* \in A \cap B \Leftrightarrow (x^*,x^*) \in X \cap Y.
    \end{equation}
     From \cite[Theorem 3]{BCS2021}, we have $C_T$ is well-defined for solving (\ref{XY}) and finds a point \((x^*,x^*) \in X \cap Y\), i.e., $C_T$ reaches a point \(x^* \in A\cap B\) with this transformation.
\end{itemize}

In contrast to other projection methods, for instance, the Douglas-Rachford method, the fixed-point set of $C_T$ coincides with $A \cap B$.

\begin{proposition}
Let $A,B \subseteq \Hi$ be closed convex sets and \(A \cap B \neq \varnothing\), then \(\Fix C_T = A \cap B\). 
\end{proposition}

\begin{proof}
See \cite[Introduction]{BCS2018-1}.
\end{proof}

\section{Finite convergence of CRM on closed convex cones in $\R^2$}\label{F}
We present the two technical results first, followed by the main theorem.


\begin{lemma}\label{orthogonality}
Suppose $C$ is a nonempty closed convex cone in $\Hi$. Let \(x \in \Hi\) and \(p=P_C(x)\), then 
\[\langle p \,, x-p \rangle = 0.\]
\end{lemma}
\begin{proof}
Since $C$ is a nonempty closed convex cone in $\Hi$ and \(p \in C\), then \(0 \in C\) and \(2p \in C\). By Theorem \ref{ProjectionTheorem} we have
\[\langle 0-p \,, x-p \rangle \leq 0 \quad \mbox{and} \quad \langle 2p-p \,, x-p \rangle \leq 0,\]
i.e., \(\langle p \,, x-p \rangle = 0\).
\end{proof}

\begin{corollary}\label{PreserveNorm}
Suppose $C$ is a nonempty closed convex cone in $\Hi$ and $R_C$ is the reflection operator associated with $C$. Then for every \(x \in \Hi\), \(\|x\|=\|R_C(x)\|\).
\end{corollary}
\begin{proof}
Suppose \(p=P_C(x)\), then by Lemma \ref{orthogonality},
\begin{equation*}
    \|R_C(x)\|^2 =\|2p-x\|^2=4\|p\|^2 - 4\langle p \,, x \rangle + \|x\|^2=4\langle p \,, p-x \rangle+\|x\|^2=\|x\|^2,
\end{equation*}
i.e., \(\|x\|=\|R_C(x)\|\).
\end{proof}

Now we state the first main theorem of this paper.

\begin{theorem}\label{main}
Let \(A,B \subseteq \R^2\) be nonempty closed convex cones. Then for any initial guess \(x \in \R^2\), CRM finds a feasible point in $A \cap B$ in at most three steps.
\end{theorem}

\begin{proof}
First, we have \(0 \in A \cap B\) thus \(A \cap B\) is always nonempty. For simplicity, in this proof we let \(y\coloneqq R_A(x)\) and \(z\coloneqq R_B(y)=R_B(R_A(x))\). Let \(x_1\coloneqq C_T(x)\), and similarly \(y_1\coloneqq R_A(x_1),z_1\coloneqq R_B(y_1)=R_B(R_A(x_1))\) as well. The proof can be split into three parts according to the cardinality of \(\{x,y,z\}\). If the cardinality of \(\{x,y,z\}\) is 1, then \(x=y=z\), and thus \(x \in A \cap B\). If the cardinality of \(\{x, y, z\}\) is 3, then $x, y$ and $z$ are all distinct. Since reflection
operators \(R_A\) and \(R_B \circ R_A\) are norm-preserving by Corollary \ref{PreserveNorm}, we have \(\|x\|=\|y\|=\|z\|\) and \(C_T(x)=0 \in A \cap B\). The remaining case is when the cardinality of \(\{x, y, z\}\) is 2.

\begin{enumerate}[(i)]
    \item If \(x=y \neq z\), equivalently \(x \in A\) and \(x \notin B\), then by definition of \(C_T\) we have \[C_T(x)=\frac{1}{2}(x+z)=\frac{1}{2}(x+R_B(R_A(x)))=\frac{1}{2}(x+2P_B(x)-x)=P_B(x),\]
    i.e., \(x_1=C_T(x) \in B\).
    \begin{enumerate}[(a)]
        \item If \(x_1 \in A\), then \(C_T(x) \in A \cap B\) trivially. 
        \item If \(x_1 \notin A\) and \(x_1, y_1,z_1\) are distinct, then \(C_T^2(x)=C_T(x_1)=0 \in A \cap B\).
        \item If \(x_1 \notin A\) and the cardinality of \(\{x_1,y_1,z_1\}\) is 2, then either \(y_1=z_1\) or \(x_1=z_1\). Note that since \(x_1 \notin A\) we always have \(x_1 \neq y_1\). Further, the cardinality of \(\{x_1,y_1,z_1\}\) cannot be 1 otherwise \(x_1\) must lie in $A$, contradicting $x_1 \notin A$. First assume \(y_1=z_1\). As both \(x_1, y_1 \in B\) and \(B\) is a closed convex cone, we must have \[C_T(x_1)=\frac{1}{2}(x_1+ y_1)=P_A(x_1) \in B,\] 
        thus \(C_T(x_1) \in A \cap B\). Secondly, if \(x_1=z_1\), \(R_A\) and \(R_B \circ R_A\) reflect \(x_1\) through same point \(q = P_A(x_1) = P_B(y_1)\in A\cap B\). Therefore \[C_T(x_1)=\frac{1}{2}(x_1+ y_1)=\frac{1}{2}(y_1+z_1)=q\in A \cap B.\]
    \end{enumerate}
    \item If \(x \neq y=z\), equivalently \(x \notin A\) and \(y \in B\), then by definition of \(C_T\) we have
    \[C_T(x)=\frac{1}{2}(x+y)=\frac{1}{2}(x+R_A(x))=\frac{1}{2}(x+2P_A(x)-x)=P_A(x),\]
    i.e., \(C_T(x) \in A\). Then \(y_1=x_1\) thus
    \[C_T(x_1)=\frac{1}{2}(x_1+R_B(R_A(x_1)))=\frac{1}{2}(x_1+R_B(x_1))=P_B(x_1) \in B,\]
    then refer to (i).
    \item If \(x = z \neq y\), equivalently \(x \notin A\) and \(x = R_B(R_A(x))\), then by definition of \(C_T\) we have
    \[C_T(x)=\frac{1}{2}(y+z)=\frac{1}{2}(R_A(x)+R_B(R_A(x)))=P_B(R_A(x)) \in B,\]
    then refer to (i).
\end{enumerate}
\end{proof}

\section{Finite convergence of CRM on polyhedral sets}\label{P}

In this section, we apply Theorem \ref{main} to polyhedral sets in $\R^2$ and wedge-like sets in $\R^n$. Following \cite{rockafellar1970convex}, a subset of $\R^n$ is a \textit{polyhedral set} if it is a finite intersection of closed half-spaces. According to \cite[Lemma 3]{soltan2022local}, a polyhedral set is conic at every point within itself. For the convenience of the readers, we provide a definition of local conicity and a proof of the conicity of polyhedral sets.

\begin{definition}
Let $S$ be a closed convex subset of $\R^n$. For a given $x \in S$ and $r>0$, let
\[K\coloneqq (S \cap B_r(x))-x.\]
We say $S$ is \textit{locally conic} at a point $x \in S$ if there exists \(r>0\) such that
\[K=B_r(0) \cap \left(\bigcup_{\lambda \geq 0} \lambda K \right).\]
\end{definition}

Intuitively, consider a disk sector $S$ shown in Figure \ref{fig:locallyconic}, $S$ is locally conic at point $A$ and $B$, but not locally conic at point $C$ and $D$.

\begin{figure}[!ht]
    \centering
    \includegraphics[width = 10cm]{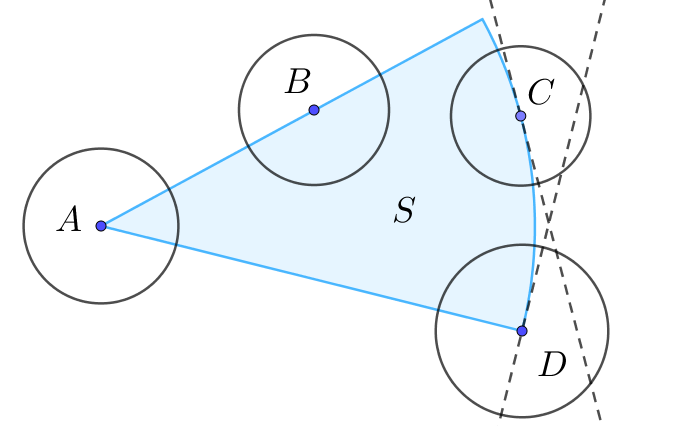}
    \caption{Geometrical interpretation of local conicity. }
    \label{fig:locallyconic}
\end{figure}

\begin{lemma}\label{PolyhedralLocallyCone}
A polyhedral set \(S \subseteq \R^n\) is locally conic at every point $x \in S$.
\end{lemma}
\begin{proof}
We know any polyhedral set $S$ can be written as
\begin{equation*}
    S=\bigcap_{i=1}^m H_i, 
\end{equation*}
where $H_i = \{x \in \R^n \,|\, \langle a_i\,,x \rangle \leq b_i\}$, \(a_1,\dots,a_m \in \R^n\) and \(a_i \neq 0\) for all \(1 \leq i \leq m\), \(b_1,\dots,b_m \in \R\). For a given point $x \in S$, if $x \in \mbox{int}\, S$, \(S\) is locally conic at $x$ trivially. Otherwise, assume $x \in \bdry S$, after relabeling the index of $H_i$'s, there exists an integer $j$ such that
\[\langle a_i\,,x \rangle = b_i,\,\forall\, 1\leq i \leq j \quad\mbox{and}\quad \langle a_i\,, x \rangle < b_i, \,\forall\, j<i \leq m.\]
Since \(m\) is finite, we can choose
\[0<r<\mbox{min}\left\{ \frac{|\langle a_i\,,\,x \rangle - b_i|}{\|a_i\|}, j<i \leq m \right\},\]
thus for all \( y \in S \cap B_r(x)\), \(\langle a_i\,,\,y \rangle < b_i\) for all \( j<i \leq m\). Further, if we let $C = S \cap \bdry B_r(x)$, we can conclude
\begin{equation}\label{BrxS}
    S \cap B_r(x) = \bigcup_{y \in C} [x,y].
\end{equation}
Let \(K=(S \cap B_r(x))-x\). It is sufficient to check whether \(K \supseteq B_r(0) \cap (\bigcup_{\lambda \geq 0} \lambda K)\). Indeed, by (\ref{BrxS}) we have $\bigcup_{0 \leq \lambda \leq 1} \lambda K = K \subseteq B_r(0)$, and \(B_r(0) \cap (\bigcup_{\lambda > 1} \lambda K) = B_r(0) \cap K\). Therefore, \(S\) is locally conic at $x$.
\end{proof}

Now we are ready to apply Theorem \ref{main} to polyhedral sets in $\R^2$ and analyze the convergence performance of CRM.

\begin{theorem}\label{poly_finite}
Suppose \(A,B \subseteq \R^2\) are polyhedral sets and \(A \cap B \neq \varnothing\). Then for any initial guess \(x \in \R^2\), the sequence generated by CRM converges to a point in \(A \cap B\) finitely.
\end{theorem}
\begin{proof}
Since $A, B$ are polyhedral sets, $A \cap B$ is a polyhedral set as well. By \cite[Theorem 3]{BCS2021}, for a given \(x \in \R^2\), let $x_n = C_T^n(x)$, then sequence \(\{x_n\}\) converges to a point \(x^* \in A \cap B\). According to Lemma \ref{PolyhedralLocallyCone}, \(A \cap B\) is locally conic at \(x^*\). Suppose the associated radius is \(r\). Then there exists a finite \(N \in \mathbb{N}\) such that \(x_n \in B_r(x^*)\) for all \(n>N\). From this stage, we have 
\[R_A(x_n) = R_{A\cap B_r(x^*)}(x_n) \quad\mbox{and}\quad R_B(R_A(x_n)) = R_{B\cap B_r(x^*)}(R_{A\cap B_r(x^*)}(x_n)).\]
In other word, CRM iterations for $A\cap B$ and $(A\cap B_r(x^*))\cap (B\cap B_r(x^*))$ are indistinguishable. Consider the problem:
\begin{equation*}
    \mbox{find}\; y \in ((A\cap B_r(x^*))-x^*) \cap ((B\cap B_r(x^*))-x^*).
\end{equation*}
Clearly, both sets \(((A\cap B_r(x^*))-x^*)\) and \(((B\cap B_r(x^*))-x^*)\) are closed, convex, and locally conic at 0. By Theorem \ref{main}, we  conclude that \(\{x_n\}\) converges to \(x^*\) in no more than \(N+3\) steps.
\end{proof}

The above theorem can be extended to higher dimensional cases if all the reflections are in the same plane.


\begin{corollary}\label{cor:wedge1}
    Suppose $A,B \subseteq \R^n$, $n \geq 3$ and there is a two-dimensional subspace $L_0 \subseteq \R^n$ with
    \[A = A_0 \oplus M \quad\mbox{and}\quad B = B_0 \oplus M,\]
    where $A_0,B_0$ are polyhedral sets in $L_0$, $A_0 \cap B_0 \neq \varnothing$, and $M$ is a linear subspace of $\R^n$ such that $M \perp L$. Then for any initial guess \(x \in \R^n\), the sequence generated by CRM converges to a point in \(A \cap B\) finitely.
\end{corollary}
\begin{proof}
    For a given \(x \in \R^n\), there always exists a two-dimensional subspace $L \subseteq \R^n$ which contains $x$ and $L \perp M$. By \cite[Corollary 3.22]{BauschkeHeinz2017}, for any $u,v \in M$, we have
    \[\langle u-v \,,\, x-P_{M}(x)\rangle = 0,\]
    i.e., $x-P_{M}(x)$ is perpendicular to $M$, $P_{M}(x) \in L$. Let $A_L = A \cap L$ and $B_L = B\cap L$, we can conclude that $A_L, B_L$ are polyhedral sets in $L$. Let $x=(x_L,x_M)$, where $x_L \in L$ and $x_M \in M$, we have $P_A(x) = (P_{A_L}(x), P_{M}(x))$. Since $P_{M}(x) \in L$, we have $P_A(x) = (P_{A_L}(x), x_M) \in L$. Similarly, $P_B(x) = (P_{B_L}(x), x_M) \in L$ as well. Therefore, $A_L, B_L, x, P_A(x), P_B(x)$ lie in the same two-dimensional subspace $L$. By Theorem~\ref{poly_finite}, we can conclude that $x_L$ converges to a point in $A_L \cap B_L$ finitely. Therefore, CRM finds a point in $A\cap B$ from $x$ finitely.
\end{proof}

Further, the direct sum can be relaxed to the Minkowski sum. 

\begin{corollary}\label{cor:wedge2}
    Suppose $A,B \subseteq \R^n$ are full dimensional, $A\cap B \neq \varnothing$ and $n \geq 3$. If there exists a linear subspace $M$ of $\R^n$ with dimension $n-2$ such that
    \[A = A_0 + M \quad\mbox{and}\quad B = B_0 + M,\]
    where $A_0, B_0$ are two-dimensional polyhedral sets, $A_0, B_0 \nsubseteq M$. Then for any initial guess \(x \in \R^n\), the sequence generated by CRM converges to a point in \(A \cap B\) finitely.
\end{corollary}
\begin{proof}
    According to the conditions of $A,B$, we can conclude there exists a two-dimensional subspace $L_0 \subseteq \R^n$ such that $L_0 \perp M$. Let $A_1 = A \cap L_0$ and $B_1 = B \cap L_0$, we can write
    \[A = A_1 \oplus M \quad\mbox{and}\quad B = B_1 \oplus M,\]
    where $A_1, B_1 \subseteq L_0$ are polyhedral sets, $A_1 \cap B_1 \neq \varnothing$. Then the statement follows from Corollary~\ref{cor:wedge1}.
\end{proof}

\section{Finite convergence of CRM on proper polyhedral cones in \(\R^3\)} \label{FR3}


In this section, we examine the convergence properties of two proper polyhedral cones in $\R^3$. A cone $C \subseteq \mathbb{R}^n$ is called \textit{proper} if it is convex, closed, pointed (i.e., $C \cap -C = \{0\}$) and solid (i.e., $\mbox{int}\,C \neq \varnothing$). To analyze these properties, we project the proper polyhedral cones onto 
$\mathbb{S}^2$ and adapt the method used in the previous two sections.

\subsection{From $\R^3$ to $\mathbb{S}^2$}

We first establish a connection between the unit sphere \(\mathbb{S}^2\) and \(\mathbb{R}^3\). Let \(A' = A \cap \mathbb{S}^2\) and \(B' = B \cap \mathbb{S}^2\). To perform iterations with respect to $A'$ and $B'$, we modify CRM into a new technique that we call the Sphere-Centered Reflection Method (SRM), which allows us to replicate CRM's performance on \(\mathbb{S}^2\). We then examine the convergence properties of SRM and relate them to CRM.

\par In the following, we first present the preliminaries leading to the development of SRM. A set \(H_S\) is called a \textit{half-sphere} of \(\mathbb{S}^2\) if there exists a closed half-space \(H \subseteq \mathbb{R}^3\) such that \(0 \in H\) and \(H_S = H \cap \mathbb{S}^2\). For points \(x, y \in \mathbb{S}^2\), the \textit{geodesic} is the shorter arc of the great circle passing through \(x\) and \(y\), denoted \(\geo(x, y)\). If \(x\) and \(y\) are antipodal points, there are infinitely many geodesics between them, and \(\geo(x, y)\) represents the set containing all such geodesics. The length of a geodesic is \(|\geo(x, y)|\). Similar to $\R^3$, the \textit{geodesic distance} between points $x$ and $y$ in $\mathbb{S}^2$ is
\[d^g(x,y) \coloneqq \arccos\, \langle x,y\rangle = |\geo (x,y)|,\]
Recall that there is a classic metric exists on the sphere with geodesics:


\begin{proposition}
    Let $\mathbb{S}^2$ be the unit sphere, $x,y \in \mathbb{S}^2$. Suppose $d^g(x,y)$ is the geodesic distance between $x$ and $y$, then $(\mathbb{S}^2,d^g)$ is a metric space.
\end{proposition}

\begin{proof}
See \cite[Chapter 1 \& 5]{petersen2006riemannian}.
\end{proof}

Before introducing SRM, we first define convex sets and polyhedral sets in \(\mathbb{S}^2\), followed by the definitions of the projection and reflection operators in \(\mathbb{S}^2\).

\par Let \(C \subseteq \mathbb{S}^2\) be a nonempty set. The set \(C\) is called \textit{geodesically convex} in \(\mathbb{S}^2\) if, for all \(x, y \in C\), there exists a unique \(\geo(x, y) \subseteq C\). Furthermore, \(C\) is a \textit{polyhedral set in} \(\mathbb{S}^2\) if it can be expressed as the intersection of a finite number of half-spheres in \(\mathbb{S}^2\). With these definitions, we can now describe the relationship between a proper polyhedral cone in \(\mathbb{R}^3\) and a geodesically convex polyhedral set in \(\mathbb{S}^2\).





\begin{lemma}\label{CprimePoly}
    Suppose $C \subseteq \R^3$ is a proper polyhedral cone and let $C' = C\cap \mathbb{S}^2$. Then $C'$ is a geodesically convex polyhedral set in $\mathbb{S}^2$.
\end{lemma}
\begin{proof}
    Suppose $x,y \in C' \subseteq C$, $x \neq y$. Since $C$ is proper and thus pointed, $x$ and $y$ cannot be antipodal points. Since $C$ is proper and thus convex, the line segment $[x,y] \subseteq C$. Let $H_{x,y}$ be the plane passing through $x,y$ and the origin, then $H_{x,y} \cap \mathbb{S}^2$ is a great circle of $\mathbb{S}^2$. Therefore, $\geo (x,y)  \subseteq H_{x,y} \cap C$, and then $\geo (x,y) \subseteq C'$, $C'$ is geodesically convex. 
    \par To see $C'$ is a polyhedral set in $\mathbb{S}^2$, since $C$ can be written as
    \[C = \bigcap_{i=1}^n H_i, \;\mbox{where}\; H_i = \{x \in \R^3 \,|\, \langle a_i,\,x \rangle \leq 0 \}, a_i \in \R^3, a_i \neq 0, n \geq 3,\]
    each $H_i \cap \mathbb{S}^2$ is a closed half-sphere of $\mathbb{S}^2$, here we require $n \geq 3$ to make sure $C$ is pointed. After taking finite intersection, we can see $C'$ is a polyhedral set in $\mathbb{S}^2$.
\end{proof}

To transfer from $\R^3$ to $\mathbb{S}^2$, we need to redefine the projection in $(\mathbb{S}^2,d^g)$. Let set \(C \subseteq \mathbb{S}^2\) be nonempty and \(x \in \mathbb{S}^2\). The \textit{distance} of $x$ to $C$ in $\mathbb{S}^2$ is 
\[d_C^g (x)\coloneqq{\rm inf}_{y \in C} \, d^g(x,y).\]

Suppose \(C \subseteq \mathbb{R}^3\) is a proper polyhedral cone, and let \(C' = C \cap \mathbb{S}^2\). If \(x \in (\mathbb{R}^3 \setminus C^{\ominus}) \cap \mathbb{S}^2\), let \(p = P_C(x)\). Then \(p' \coloneqq [p] \cap \mathbb{S}^2 \in C'\) is the \textit{projection} of \(x\) onto \(C'\) in \(\mathbb{S}^2\), and we denote this projection by \(P^g_{C'}(x)\). The \textit{reflection} associated with \(C'\) in \(\mathbb{S}^2\) is then defined as \(R_{C'}^g \coloneqq R_C\).

\par It is worth mentioning that here we define the reflection operator \(R_{C'}^g\) independently of \(P^g_{C'}\). According to Corollary~\ref{PreserveNorm}, reflection over a cone in \(\mathbb{R}^3\) preserves the norm, allowing us to conclude that \(R_{C'}^g(x) \in \mathbb{S}^2\) when \(x \in (\mathbb{R}^3 \setminus C^{\ominus}) \cap \mathbb{S}^2\). To conclude this subsection, we show that in \(\mathbb{S}^2\), the distance between \(x\) and \(P_{C'}^g(x)\) is equal to the distance between \(P_{C'}^g(x)\) and \(R_{C'}^g(x)\), mirroring the behavior in \(\mathbb{R}^3\).

\begin{proposition}\label{dPCdRC}
    Suppose $C \subseteq \R^3$ is a proper polyhedral cone and let $C' = C\cap \mathbb{S}^2$. Suppose $x \in (\R^3 \setminus C^{\ominus}) \cap \mathbb{S}^2$, then $d^g(P^g_{C'}(x), R_{C'}^g(x))= d^g(x,P^g_{C'}(x)) = d_{C'}^g (x)$.
\end{proposition}
\begin{proof}
    The proof can be divided into two parts. To prove the first equality, suppose $p = P_C(x)$, and let $p' = P^g_{C'}(x)$ and $y = R_C(x) = R_{C'}^g(x)$ for convenience. According to Lemma~\ref{orthogonality}, we have $\langle p \,, x-p \rangle = 0$, and can derive $\langle p \,, y-p \rangle = 0$ as well. Since $x,p$ and $y$ are colinear and $p'= [p] \cap \mathbb{S}^2$, then $0,x,y$ and $p'$ are coplanar, $\| x-p' \| = \| y-p' \|$. Since $x,y,p' \in \mathbb{S}^2$, we have $|\geo (x,p')| = |\geo (p',y)|$ as well, i.e., $d^g(x,P^g_{C'}(x)) = d^g(P^g_{C'}(x), R_{C'}^g(x))$.
    
    \par We prove the second equality by contradiction. If not, suppose $d_{C'}^g (x) = d^g (x,q')$, where $q' \in C'$ and $q' \neq P^g_{C'}(x)$. Let $p = P_C(x)$, $p' = P^g_{C'}(x)$, $q=P_{[q']}(x)$, $y = R_C(x) = R_{C'}^g(x)$, $z = R_{[q']}(x) = R_{\{q'\}}^g(x)$. Since $d_{C'}^g (x) = d^g (x,q')$, then by the first part of the proof,
    \[|\geo(x,z)| = 2|\geo(x,q')| \leq 2|\geo(x,p')| = |\geo(x,y)|,\]
    therefore
    \[\|x-z\| = 2\|x-q\| \leq 2\|x-p\| = \|x-y\|.\]
    However, since $p = P_C(x)$ and $q=P_{[q']}(x) \in C$, we must have $\|x-p\| \leq \|x-q\|$ and thus $p=q$. Therefore, we must have $q'=p'=P^g_{C'}(x)$ as well.
\end{proof}


\subsection{Sphere-centered reflection method (SRM)}

In this subsection, we introduce the Sphere-Centered Reflection Method. We begin by defining the sphere-center in \(\mathbb{S}^2\), followed by the definition of SRM.

\begin{definition}
    Suppose $x,y,z \in \mathbb{S}^2$ and they are not in the same great circle of $\mathbb{S}^2$. The \textit{sphere-center} of $x,y$ and $z$, which denoted as $S(x,y,z)$, satisfies
    \begin{itemize}
    \item \(S(x,y,z) \in V\coloneqq\{p \in \mathbb{S}^2\,|\, d^g(p,x) = d^g(p,y) = d^g(p,z)\},\) and
    \item $d^g(S(x,y,z),x) \leq d^g(p,x)$ for all $p \in V$.
\end{itemize}
\end{definition}

\par
Now we give the definition of SRM.
\begin{definition}\label{SRM}
    Suppose \(A,B \subseteq \R^3\) are proper polyhedral cones. Let $A'=A\cap \mathbb{S}^2$ and $B'= B \cap \mathbb{S}^2$. Let \(R^g_{A'}, R^g_{B'}\) be reflection operators associated with \(A', B'\) respectively in \(\mathbb{S}^2\). The \textit{sphere-centered reflection operator} $S_T$ is defined as
\begin{equation*}
    S_T(x)\coloneqq S(x,y,z),
\end{equation*}
where $x\in \mathbb{S}^2$ and $x\notin \Ker C_T \cap \mathbb{S}^2$, $y = R^g_{A'}(x), z = R^g_{B'}(R^g_{A'}(x))$.
\end{definition}

Note that $S_T(x)$ is well-defined when $x \notin \Ker C_T \cap \mathbb{S}^2$. Indeed, since $A^{\ominus}, B^{\ominus} \subseteq (A\cap B)^{\ominus}$ and $(A\cap B)^{\ominus} \subseteq \Ker C_T$, reflection operators \(R^g_{A'}, R^g_{B'}\) are well-defined by definition of reflection operators in $\mathbb{S}^2$. And since $x\notin \Ker C_T \cap \mathbb{S}^2$, $x,y$ and $z$ are not in the same great circle of $\mathbb{S}^2$ and then $S_T(x)$ exists.

\subsection{Locally finite convergence of SRM on geodesically convex polyhedral sets in $\mathbb{S}^2$ }

To describe a geodesically convex polyhedral set more precisely, suppose $C \subseteq \R^3$ is a proper polyhedral cone and let $C' = C\cap \mathbb{S}^2$. By the proof of Lemma \ref{CprimePoly}, there exists the smallest $n$ such that $C$ can be written as
    \[C = \bigcap_{i=1}^n H_i, \;\mbox{where}\; H_i = \{x \in \R^3 \,|\, \langle a_i,\,x \rangle \leq 0 \}, a_i \in \R^3, a_i \neq 0, n \geq 3,\]
here each $H_i$ is a plane passing through the origin. Let $H_i' =  H_i \cap \mathbb{S}^2$ be a closed half-sphere of $\mathbb{S}^2$, then $C'$ can be written as
\[C' = \bigcap_{i=1}^n H_i', \;\mbox{where}\; n \geq 3.\]
We define the set 
\[\overline{G}_{C'}\coloneqq \{H_1',\dots,H_n'\},\] 
where each $H_i'$ is a closed half-sphere. In addition, the geodesically convex polyhedral set $C'$ can also be written as
\[C' = \conv \{V_1,\dots,V_n\},\]
where $V_1,\dots,V_n$ are the vertices of \(C'\), and $\geo(V_1, V_2),\dots, \geo(V_{n-1},V_n), \geo(V_n, V_1)$ are the edges of $C'$. We can see that each edge of $C'$ can be represented as $-H_i' \cap C'$, for one $i \in \{1,\dots,n\}$. Therefore, the set
\[G_{C'} \coloneqq \{-H_1' \cap C',\dots,-H_n' \cap C'\}\]
is the collection of all the edges of $C'$, and each element of $G_{C'}$ is a geodesic.

\par From now on to the end of the paper, we assume \(A,B \subseteq \R^3\) are proper polyhedral cones and $A \cap B$ is nontrivial, i.e., $\{0\} \subsetneq A \cap B$. Let 
\begin{equation}\label{A'B'}
    A'=A\cap \mathbb{S}^2 \quad \mbox{and} \quad B'= B \cap \mathbb{S}^2,
\end{equation}
we can conclude that $A' \cap B' \neq \varnothing$. Now we apply SRM to $A'$ and $B'$.

\begin{proposition}\label{SRM_poly}
    Suppose sets $A', B'$ are defined in (\ref{A'B'}). There exists $r>0$ such that for any initial guess $x$ that satisfies $d_{A' \cap B'}^g (x) \leq r$ and for which $S_T(x)$ is well-defined, the sequence generated by SRM converges a point within \(A' \cap B'\) in no more than three steps. 
\end{proposition}
\begin{proof}
    We know $A' \cap B' \neq \varnothing$. Let $\mathbb{B}_r(x)$ denote the closed ball centred at $x$ with radius $r$ in $\mathbb{S}^2$, i.e.,
    \[\mathbb{B}_r(x) \coloneqq \{w \in \mathbb{S}^2 \,|\, d^g(x,w) \leq r\}. \]
    \begin{itemize}
        \item We first show how to determine $r$. Since $A'$ and $B'$ are geodesically convex polyhedral sets, this makes $A' \cap B'$ a geodesically convex polyhedral set as well. Suppose $\overline{G}_{A'}, \overline{G}_{B'}$ is the least set of closed half-spheres that defines $A'$ and $B'$ respectively. Let $G_{A',B'} = G_{A'} \cup G_{B'}$, therefore $G_{A',B'}$ defines $A' \cap B'$ but not necessarily least, $G_{A'\cap B'} \subseteq G_{A',B'}$. We write $A' \cap B'$ as
        \[A' \cap B' = \conv \{V_1,\dots,V_n\},\]
        where $V_1,\dots,V_n$ are the vertices of \(A' \cap B'\). For each $V_t$, there exists a set of active geodesics $G_{V_t}^+ \subseteq G_{A',B'}$ such that for any geodesic $G \in G_{V_t}^+$, we have $V_t \in G$. Let
        \[r_t = \mbox{min}\, \{d_G^g(V_t) \,|\, G \in G_{A',B'} \setminus G_{V_t}^+\}.\]
        By definition, we can see $d_G^g(V_t) = 0$ if and only if $G \in G_{V_t}^+$, therefore $r_t >0$. Let $t$ range from $1$ to $n$ and lastly let
        \[r = \mbox{min}\, \{r_1,\dots,r_n\},\]
        this is the $r$ we want to find. In the following, we will use the same letter $r$ for convenience.
        
        \item Now we show the statement. Based on the position of $x$, the proof can be divided into two sections. Suppose the initial guess $x \notin A' \cap B'$, otherwise $x$ is already a solution. For simplicity, in the following we let \(y = R^g_{A'}(x)\) and \(z = R^g_{B'}(R^g_{A'}(x))\). Let \(x_1 = S_T(x)\), and similarly \(y_1 = R^g_{A'}(x_1), z_1 = R^g_{B'}(R^g_{A'}(x_1))\) as well. 
        \begin{enumerate}
            \item If $x$ satisfies $d_{A' \cap B'}^g (x) \leq r$ and $x \notin \mathbb{B}_r(V_t)$ for all $t=1,\dots,n$, i.e., $d^g(x,V_t) >r$ for any $t=1,\dots,n$, we can conclude that $P_{A' \cap B'}^g (x)$ is projected to an edge of $A' \cap B'$. We denote the edge as $E$. Then
            \begin{enumerate}
                \item  either $x \in A$, makes $x=y$ and $z = R_{E}^g (x)$, following $S_T(x) = \frac{1}{2}(x+z) = P_E^g (x) \in A' \cap B'$;
                \item or $x \in B$, makes $y=R_E^g (x)$ and $z = R_E^g(y) = R_E^g(R_E^g (x)) = x$, following $S_T(x) = \frac{1}{2}(x+y) = P_E^g (x) \in A' \cap B'$.
            \end{enumerate}
            That is, $x$ converges to a point in $A' \cap B'$ in one step.

        \item If $x\in \mathbb{B}_r(V_t)$ for some $t$, the iterations is analogous to the proof of Theorem~\ref{main}. We can conclude that $x$ converges to a point in $A' \cap B'$ in at most three steps.

        \end{enumerate}
    \end{itemize}

\end{proof}

\subsection{Back to $\mathbb{R}^3$ from $\mathbb{S}^2$ }

In this subsection, we characterize the convergence properties of two proper polyhedral cones in \(\mathbb{R}^3\), building on the results established in the previous subsections.

\begin{theorem}\label{thm:R3poly}
    Suppose \(A,B \subseteq \R^3\) are proper polyhedral cones and 
$A \cap B$ is nontrivial, $A', B'$ are defined in (\ref{A'B'}). Let
\[D = \{x \in \mathbb{S}^2 \,|\, d_{A' \cap B'}^g (x) \leq r\},\]
where $r$ is chosen as in Proposition~\ref{SRM_poly}. If the initial guess $x \in \Ker C_T \cup \cone (D)$, CRM locates a feasible point in $A \cap B$ in at most three steps.
\end{theorem}

\begin{proof}
    It is only necessary to check the case when $x \in \cone (D)$. Since $x \in \cone (D)$, we have $R_A([x]) = [R_A(x)]$ and $R_B([x]) = [R_B(x)]$. Without lose of generality, suppose $x \neq 0$ and let $\hat{x}\coloneqq \frac{x}{\|x\|} = [x]\cap \mathbb{S}^2$. Since $A \cap B$ is nontrivial, $A' \cap B' \neq \varnothing$. By Definition of reflection operators in $\mathbb{S}^2$, we have
    
        \[R_A ([x_n]) = [R_{A'}^g (\hat{x}_n)] \quad\mbox{and}\quad R_B(R_A ([x_n])) = [R_{B'}^g(R_{A'}^g (\hat{x}_n))].\]
        Combining with Corollary~\ref{PreserveNorm}, we can conclude that
        \[C_T([x_n]) = [C_T(x_n)] = [S_T(\hat{x}_n)].\]
        According to Proposition~\ref{SRM_poly}, $\hat{x}_n$ will converge to a point $\hat{x}^* \in A' \cap B'$ in no more than three steps, means $[x_n] \to [x^*] \subseteq A\cap B$. Namely, CRM finds a nonzero feasibility point in $A \cap B$ in at most 3 steps.
\end{proof}

To conclude this paper, we present an example to show that if the initial guess is outside of the zone determined in Theorem \ref{thm:R3poly}, the finite convergence of CRM may fail.

\begin{example}\label{counterexample}
    Let
    \[a_1 = (3,0,3), a_2 = (0,1,3), a_3 = (0,-1,3), a_4 = (-3,0,-2),\]
    and
    \[b_1 = (1,3,0), b_2 = (1,-3,0), b_3 = (-3,0,-1).\]
    Suppose
    \[A = \cone (\{a_1, a_2, a_3, a_4\}) \quad \mbox{and} \quad B = \cone (\{b_1,b_2,b_3\}).\]
    Sets $A, B$ are proper polyhedral cones in $\R^3$, $A\cap B$ is non-trivial, which is shown in Figure \ref{fig:coneAB}. There exists initial guess $x^0 \in \R^3$ such that sequence $\{C_T^n(x^0)\}$ does not converge to a feasible point in $A\cap B$ finitely.
\end{example}

\begin{figure}[!ht]
    \centering
    \includegraphics[width = 10cm]{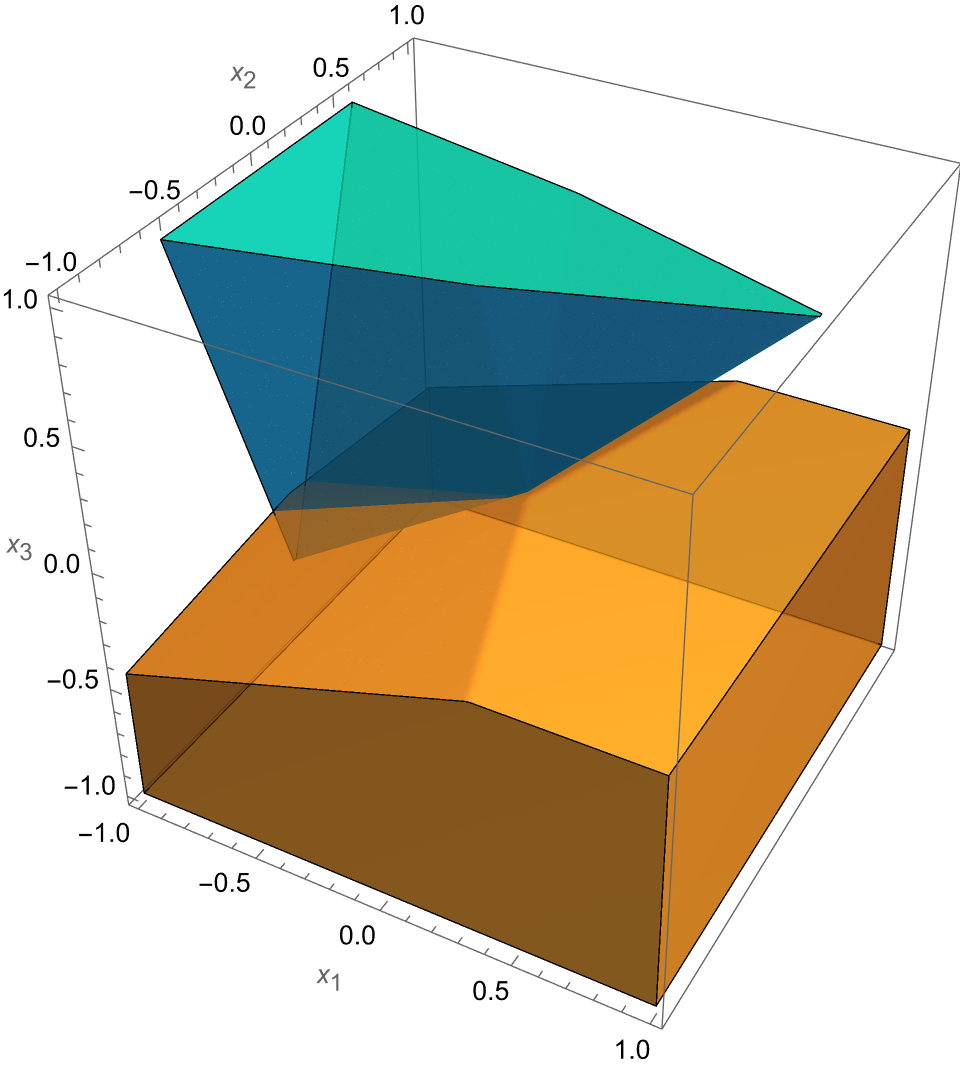}
    \caption{The figure of cone $A$ (cyan) and $B$ (golden).}
    \label{fig:coneAB}
\end{figure}

\begin{proof}
    Take $x^0 = (x_1^0,x_2^0,0)$, where $0<|x_2^0|<x_1^0$. Let $L=\{(s,0,s) \in \R^3 \,|\, s \geq 0\}$, we claim
    \[P_A(x^0) = P_L (x^0) = \left(\frac{x_1^0}{2},0,\frac{x_1^0}{2} \right) \quad \mbox{and} \quad R_A(x^0) = 2P_A(x^0) - x^0 = (0,-x_2^0, x_1^0).\]
    Indeed, according to the definition of conic hull, all points $y \in A$ can be represented as 
    \begin{align*}
        y & = \alpha_1 (3,0,3) + \alpha_2 (0,1,3) + \alpha_3 (0,-1,3) + \alpha_4 (-3,0,-2) \\
        & = (3(\alpha_1-\alpha_4),\, \alpha_2-\alpha_3,\, 3(\alpha_1+\alpha_2+\alpha_3) - 2\alpha_4),
    \end{align*}
    where $\alpha_1, \alpha_2, \alpha_3,\alpha_4 \geq 0$. Let $p = P_A(x^0)$ and for any $y \in A$, 
    \begin{align*}
        \langle y-p, x^0-p \rangle & = \left(3(\alpha_1-\alpha_4) - \frac{x_1^0}{2},\, \alpha_2-\alpha_3,\, 3(\alpha_1+\alpha_2+\alpha_3) - 2\alpha_4 -\frac{x_1^0}{2} \right) \\
        & \;\; \cdot \left(\frac{x_1^0}{2}, x_2^0, -\frac{x_1^0}{2} \right)\\
        & = x_2^0 (\alpha_2-\alpha_3) - \frac{3 x_1^0}{2}(\alpha_2+\alpha_3) - \frac{x_1^0}{2}\alpha_4 \\ 
        & \leq |x_2^0|(\alpha_2+\alpha_3) - \frac{3 x_1^0}{2}(\alpha_2+\alpha_3) - \frac{x_1^0}{2}\alpha_4 \\
        & \leq (x_1^0 - \frac{3 x_1^0}{2} )(\alpha_2+\alpha_3) - \frac{x_1^0}{2}\alpha_4 \\ 
        & \leq 0.
    \end{align*}
    By Theorem~\ref{ProjectionTheorem}, $p$ is the projection of $x^0$ to $A$. Next, we claim
    \[P_B(R_A(x^0)) = (0,-x_2^0,0) \quad \mbox{and} \quad R_B(R_A(x^0)) = (0,-x_2^0,-x_1^0).\]
    Indeed, all points $y \in B$ can be represented as 
    \[y = \beta_1 (1,3,0) + \beta_2 (1,-3,0) + \beta_3 (-3,0,-1) = (\beta_1+\beta_2-3\beta_3,\, 3(\beta_1-\beta_2),\, -\beta_3),\]
    where $\beta_1, \beta_2, \beta_3 \geq 0$. Let $q = (0,-x_2^0,0)$ and for any $y \in B$, we have
    \begin{align*}
        \langle y-q, R_A(x^0)-q \rangle & = (\beta_1+\beta_2-3\beta_3 ,\, 3(\beta_1-\beta_2)+x_2^0,\, -\beta_3) \cdot (0,0,x_1^0) = -x_1^0 \beta_3 \leq 0.
    \end{align*}
    For convenience, let $u = (0,-x_2^0, x_1^0)$ and $v = (0,-x_2^0,-x_1^0)$. Since $C_T(x^0) = C(x,u,v) \coloneqq (c_1,c_2,c_3)$, then $C_T(x^0)$ satisfies
    \begin{align*}
        \|C_T(x^0)-u\|^2 & = c_1^2 + (c_2+x_2^0)^2 + (c_3-x_1^0)^2 \\
        = \|C_T(x^0)-v\|^2 & = c_1^2 + (c_2+x_2^0)^2 + (c_3+x_1^0)^2 \\
        = \|C_T(x^0)-x^0\|^2 & = (c_1-x_1^0)^2 + (c_2-x_2^0)^2 + c_3^2.
    \end{align*}
    Since $x_1^0 \neq 0$, we can derive $c_3 =0$ and the above equations is equivalent to 
    \[c_1^2 + (c_2+x_2^0)^2 + (x_1^0)^2= (c_1-x_1^0)^2 + (c_2-x_2^0)^2, \;\;\mbox{i.e.,} \;\; -c_1 x_1^0 = 2c_2 x_2^0.\]
    Add the condition $C_T(x^0) \in \aff (x^0,u,v)$, find $C_T(x^0)$ is equivalent to solve the optimization problem:
    \[\mbox{minimize} \quad c_1^2 + (c_2+x_2^0)^2 + (x_1^0)^2\]
    \[\mbox{subject to}\quad c_1x_1^0 + 2c_2 x_2^0 = 0.\]
    Using the Lagrange multipliers, we can derive
    \[c_1 = \frac{2x_1^0 (x_2^0)^2}{4(x_2^0)^2 + (x_1^0)^2},\quad c_2 = \frac{-x_2^0 (x_1^0)^2}{4(x_2^0)^2 + (x_1^0)^2}.\]
    We can see $0<c_1<x_1^0, 0<|c_2|<|x_2^0|$, $\{C_T^n(x^0)\}$ will converge to 0. However, according to the definition of sets $A$ and $B$, $c_1 \leq c_3$ is a necessary condition to become a feasible point of $A\cap B$. Since $0 = c_3 < c_1$, $c$ cannot be in $A\cap B$, and for the same reason all finite iterations cannot be in $A\cap B$ either.
\end{proof}

\section*{Acknowledgements}
I would like to thank my supervisor, Associate Professor Vera Roshchina, for guiding me in exploring this topic and providing invaluable feedback. I am also grateful to my joint supervisor, Dr. Mareike Dressler, for her valuable suggestions and comments in refining and polishing this paper.

\bibliography{refs}

\begin{thebibliography}{10}

\bibitem{arefidamghani2021circumcentered}
R.~Arefidamghani, R.~Behling, Y.~Bello-Cruz, A.~N. Iusem, and L.-R. Santos.
\newblock The circumcentered-reflection method achieves better rates than
  alternating projections.
\newblock {\em Computational Optimization and Applications}, 79(2):507--530,
  2021.

\bibitem{iusem2022circumcentered}
R.~Arefidamghani, A.~N. Iusem, and R.~Behling.
\newblock Circumcentered-reflection methods for the convex feasibility problem
  and the common fixed-point problem for firmly nonexpansive operators.
\newblock 2022.

\bibitem{bauschke1996projection}
H.~H. Bauschke and J.~M. Borwein.
\newblock On projection algorithms for solving convex feasibility problems.
\newblock {\em SIAM review}, 38(3):367--426, 1996.

\bibitem{BauschkeHeinz2017}
H.~H. Bauschke and P.~L. Combettes.
\newblock {\em Convex analysis and monotone operator theory in {H}ilbert
  spaces}.
\newblock CMS Books in Mathematics/Ouvrages de Math\'{e}matiques de la SMC.
  Springer, Cham, second edition, 2017.
\newblock With a foreword by H\'{e}dy Attouch.

\bibitem{bauschke2020circumcentered}
H.~H. Bauschke, H.~Ouyang, and X.~Wang.
\newblock Circumcentered methods induced by isometries.
\newblock {\em Vietnam Journal of Mathematics}, 48:471--508, 2020.

\bibitem{BCS2018-1}
R.~Behling, J.~Y. Bello-Cruz, and L.-R. Santos.
\newblock Circumcentering the {D}ouglas-{R}achford method.
\newblock {\em Numer. Algorithms}, 78(3):759--776, 2018.

\bibitem{BCS2018-2}
R.~Behling, J.~Y. Bello-Cruz, and L.-R. Santos.
\newblock On the linear convergence of the circumcentered-reflection method.
\newblock {\em Oper. Res. Lett.}, 46(2):159--162, 2018.

\bibitem{BCS2020}
R.~Behling, J.~Y. Bello-Cruz, and L.-R. Santos.
\newblock The block-wise circumcentered-reflection method.
\newblock {\em Comput. Optim. Appl.}, 76(3):675--699, 2020.

\bibitem{BCS2021}
R.~Behling, J.~Y. Bello-Cruz, and L.-R. Santos.
\newblock On the circumcentered-reflection method for the convex feasibility
  problem.
\newblock {\em Numer. Algorithms}, 86(4):1475--1494, 2021.

\bibitem{combettes1993foundations}
P.~L. Combettes.
\newblock The foundations of set theoretic estimation.
\newblock {\em Proceedings of the IEEE}, 81(2):182--208, 1993.

\bibitem{dao2023douglas}
M.~N. Dao, M.~Dressler, H.~Liao, and V.~Roshchina.
\newblock Douglas--rachford is the best projection method.
\newblock {\em arXiv preprint arXiv:2310.17077}, 2023.

\bibitem{ouyang2020finite}
H.~Ouyang.
\newblock Finite convergence of locally proper circumcentered methods.
\newblock {\em arXiv preprint arXiv:2011.13512}, 2020.

\bibitem{petersen2006riemannian}
P.~Petersen.
\newblock {\em Riemannian geometry}, volume 171.
\newblock Springer, 2006.

\bibitem{rockafellar1970convex}
R.~T. Rockafellar.
\newblock {\em Convex analysis}, volume~18.
\newblock Princeton university press, 1970.

\bibitem{soltan2022local}
V.~Soltan.
\newblock Local conicity and polyhedrality of convex sets.
\newblock {\em Results in Mathematics}, 77(5):188, 2022.

\end{thebibliography}
\bibliographystyle{plain}

\end{document}